\documentclass[12pt,a4paper,final, dvipsnames,reqno]{amsart}

\usepackage{amsfonts,amsmath,amssymb,indentfirst,mathrsfs,amscd}
\usepackage{graphicx}
\usepackage{amsthm}
\usepackage{color}
\usepackage[spanish, english]{babel}
\usepackage{hyperref}
\usepackage{subcaption}
\usepackage[all]{xy}
\usepackage{stmaryrd}

\makeatletter
\@namedef{subjclassname@2020}{\textup{2020} Mathematics Subject Classification}
\makeatother

\DeclareMathOperator{\Hom}{Hom\,}           
\DeclareMathOperator{\Spec}{Spec}

\author[\'A. Gonz\'alez-Prieto]{\'Angel Gonz\'alez-Prieto}
\address{Departamento de Matem\'aticas, Facultad de Ciencias, Universidad Aut\'onoma de Madrid. C.\ Francisco Tomás y Valiente, 7, 28049 Madrid, Spain.}
\address{Instituto de Ciencias Matem\'aticas (CSIC-UAM-UC3M-UCM), C.\ Nicol\'as Cabrera 15, 28049 Madrid, Spain.}
\email{angel.gonzalezprieto@uam.es}

\author[V. Mu\~{n}oz]{Vicente Mu\~{n}oz}
\address{Departamento de \'Algebra, Geometr\'ia y Topolog\'ia, Facultad de Ciencias, Universidad de M\'alaga, Campus de Teatinos s/n, 29071 Málaga, Spain.}
\email{vicente.munoz@uma.es}

\title[The point counting problem in representation varieties]{The point counting problem in\\representation varieties of torus knots}

\keywords{Torus knots, representation varieties, affine group, finite fields}
\subjclass[2020]{14G15, 14D20, 20G15, 14C30}

\begin{document}

\newtheorem{thm}{Theorem}[section]
\newtheorem{prop}[thm]{Proposition}
\newtheorem{lem}[thm]{Lemma}
\newtheorem{cor}[thm]{Corollary}
\newtheorem{conjecture}{Conjecture}
\newtheorem*{conjecture*}{Conjecture}
\newtheorem*{theorem*}{Theorem}

\theoremstyle{definition}
\newtheorem{defn}[thm]{Definition}
\newtheorem{ex}[thm]{Example}
\newtheorem{as}{Assumption}
\theoremstyle{remark}
\newtheorem{rmk}[thm]{Remark}
\theoremstyle{remark}
\newtheorem*{prf}{Proof}

\newcommand{\cA}{\mathcal{A}}
\newcommand{\cC}{\mathcal{C}}
\newcommand{\cD}{\mathcal{D}}
\newcommand{\cE}{\mathcal{E}}
\newcommand{\cF}{\mathcal{F}}
\newcommand{\cG}{\mathcal{G}} 
\newcommand{\cI}{\mathcal{I}} 
\newcommand{\cO}{\mathcal{O}} 
\newcommand{\cM}{\mathcal{M}} 
\newcommand{\cN}{\mathcal{N}} 
\newcommand{\cP}{\mathcal{P}} 
\newcommand{\cQ}{\mathcal{Q}} 
\newcommand{\cS}{\mathcal{S}} 
\newcommand{\cU}{\mathcal{U}} 
\newcommand{\cJ}{\mathcal{J}}
\newcommand{\cX}{\mathcal{X}}
\newcommand{\cT}{\mathcal{T}}
\newcommand{\cV}{\mathcal{V}}
\newcommand{\cW}{\mathcal{W}}
\newcommand{\cB}{\mathcal{B}}
\newcommand{\cR}{\mathcal{R}}
\newcommand{\cH}{\mathcal{H}}
\newcommand{\cZ}{\mathcal{Z}}

\newcommand{\QQ}{\mathbb{Q}} 
\newcommand{\FF}{\mathbb{F}} 
\newcommand{\PP}{\mathbb{P}} 
\newcommand{\HH}{\mathbb{H}} 
\newcommand{\RR}{\mathbb{R}} 
\newcommand{\ZZ}{\mathbb{Z}} 
\newcommand{\NN}{\mathbb{N}} 
\newcommand{\DD}{\mathbb{D}} 
\newcommand{\CC}{\mathbb{C}} 

\newcommand{\bk}{\mathbf{k}}
\newcommand{\too}{\longrightarrow}
\newcommand{\G}{\Gamma}         
\newcommand{\id}{\mathrm{id}}
\newcommand{\Ker}{\textrm{Ker}\,}
\newcommand{\coker}{\textrm{coker}\,}
\newcommand{\x}{\times}
\newcommand\HS[1]{\mathbf{HS}^{#1}}
\newcommand\MHS[1]{\mathbf{MHS}}
\newcommand\Var[1]{\mathbf{Var}_{#1}}
\newcommand\K[1]{\mathrm{K}#1}
\newcommand\AGL[1]{\mathrm{AGL}_{#1}(\bk)}
\newcommand\AGLC[1]{\mathrm{AGL}_{#1}(\CC)}
\newcommand\GL[1]{\mathrm{GL}_{#1}(\bk)}
\newcommand\GLC[1]{\mathrm{GL}_{#1}(\CC)}
\newcommand\SL[1]{\mathrm{SL}_{#1}(\bk)}
\newcommand\SLC[1]{\mathrm{SL}_{#1}(\CC)}
\newcommand\PGL[1]{\mathrm{PGL}_{#1}(\bk)}
\newcommand\PGLC[1]{\mathrm{PGL}_{#1}(\CC)}
\newcommand\Rep[1]{\mathfrak{X}_{#1}}
\newcommand\Id{\mathrm{Id}}
\newcommand\Char[1]{\cR_{#1}}
\newcommand{\tr}{\textrm{tr}\,}             
\newcommand\Gr{\textrm{Gr}}
\newcommand{\red}{\mathrm{red}}
\newcommand{\irr}{\mathrm{irr}}
\newcommand{\charF}{\mathrm{char}}

\begin{abstract}
We compute the motive of the variety of representations of the torus knot of type $(m,n)$ into the affine groups $\AGL{1}$ and $\AGL{2}$ for an arbitrary field $\bk$. In the case that $\bk = \FF_q$ is a finite field this gives rise to the count of the number of points of the representation variety, while for $\bk = \CC$ this calculation returns the $E$-polynomial of the representation variety. We discuss the interplay between these two results in sight of Katz theorem that relates the point count polynomial with the $E$-polynomial. In particular, we shall show that several point count polynomials exist for these representation varieties, depending on the arithmetic between $m,n$ and the characteristic of the field, whereas only one of them agrees with the actual $E$-polynomial.
\end{abstract}
\maketitle

\vspace{-0.4cm}

{{\em Dedicated to Prof.\ Themistocles M.\ Rassias on the occasion of his 70th birthday, 
with our utmost gratitude for his kindness at all moments.}}


\vspace{-0cm}

\section{Introduction}
Given a manifold $M$, the variety of representations $\rho: \pi_1(M) \to G$ of its fundamental group into an algebraic group $G$ contains information on the topology of $M$. It is especially
interesting for $3$-dimensional manifolds, where the fundamental
group and the geometrical properties of the manifold are 
strongly related \cite{CS}. 
This can be used to study knots $K\subset S^3$, by analyzing the representation variety for $G= \SLC{2}$ of the fundamental group of the knot complement
$S^3-K$. 
In this paper we focus in the case in which $K$ is a torus knot, which is the first family of knots where the computations are
rather feasible. The geometry of $\SLC{2}$-representation varieties of torus knots as been described in \cite{MO,Mu}. For $\SLC{3}$, it has been carried out in \cite{MP}, and 
the case of $\SLC{4}$ has been addressed in \cite{GPM} through a computer-aided proof. 

The case in which $M$ is a closed orientable surface has also been extensively analyzed due to their prominent role in non-abelian Hodge theory \cite{HT,Hi}. In these cases, the approach has a geometric flavor focused on obtaining explicit descriptions of algebro-geometric invariants of the representation variety. This is at the heart of much recent research that justifies the study of the geometry of representation varieties of surface groups, 
in particular their Hodge numbers and $E$-polynomials (defined in Section \ref{sec:Hodge}).

An existing approach to address this problem is the so-called \emph{geometric method}, initiated by  Logares, Mu\~noz and Newstead in \cite{LMN}, which aims to compute $E$-polynomials of representation varieties of surface groups. In this method, the representation variety is chopped into 
simpler strata for which the $E$-polynomial can be computed. Following this idea, in the case $G=\SLC2$, 
the $E$-polynomials were computed in a series of papers \cite{LMN,MM,MM:2016} and for
$\PGLC2$ in \cite{Martinez:2017}.

Additionally, in the papers \cite{LMN,MM}, the authors showed that a recursive pattern underlies the computations. This led to another approach, the \emph{quantum method}, initiated by Gonz\'alez-Prieto, Logares and Mu\~noz in \cite{GPLM-2017}. This method is based on the existence of a Topological Quantum Field Theory (TQFT) that provides a powerful machinery to compute $E$-polynomials of
representation varieties. Furthermore, the quantum method also enables the computation of the motive of the representation variety in the Grothendieck ring of algebraic varieties, a subtler invariant than the $E$-polynomial that will be of interest for the purposes of this work 
\cite{GP:2018a,GPLM-2017}.

A different landscape of the problem is drawn by the third method present in the literature, the \emph{arithmetic method}. This approach,
introduced by Hausel and Rodr\'iguez-Villegas in \cite{Hausel-Rodriguez-Villegas:2008}, is inspired in the Weil conjectures and aims to compute the number of points of the representation variety over finite fields. In \cite{Hausel-Rodriguez-Villegas:2008} the authors obtained the $E$-polynomials for $\GLC{r}$ in terms of generating functions, and Mereb \cite{mereb} studied this case for $\SLC{r}$, giving
an explicit formula for the $E$-polynomial in the case $\SLC{2}$.
Moreover, using this technique, explicit expressions of the $E$-polynomials have been computed \cite{Baraglia-Hekmati:2016}
for orientable surfaces with $G= \GLC{3}$, $\SLC{3}$ and for non-orientable surfaces with $G= \GLC2$, $\SLC2$.  

The arithmetic method deeply relies in a theorem of Katz contained in \cite{Hausel-Rodriguez-Villegas:2008}, 
that provides a link between arithmetic geometry and complex geometry. Let us describe briefly the main ideas of this result (for a more detailed account, refer to \cite{Hausel-Rodriguez-Villegas:2008}). 
Let $X$ be a complex variety and suppose that it makes sense to consider $X$ also over finite fields $X(\FF_q)$ (for instance, because the equations defining $X$ have integral coefficients). Hence, $X(\FF_q)$ is a finite collection of points and we can count them. Suppose moreover that there exists a polynomial $P(t) \in \ZZ[t]$ such that $P(q) = |X(\FF_q)|$ for all large enough $q$. In this case, we will say that $X$ is \emph{asymptotically polynomial count} and the polynomial $P(t)$ will be called the \emph{counting polynomial}. In that case, Katz theorem proves that the $E$-polynomial of the complex variety $X$ is actually a polynomial in the product variable $q=uv$ and coincides with the counting polynomial. In this way, the arithmetic approach is indeed a method of computing solutions of equations on finite fields.

In this work, we focus on the group $G=\AGL{r}$ of affine automorphisms of the $r$-dimensional affine space 
$\bk^r$ over an arbitrary field $\bk$ (possibly not algebraically closed). We will study the space of $\AGL{r}$-representations of torus knots of type $(m,n)$, 
denoted $\Rep{m,n}(\AGL{r})$, and we shall provide an explicit computation of their motives for ranks $r = 1, 2$. This result extends the previous work \cite{proceedings-Rassias} where the motive of the representation variety in the case $\bk = \CC$ was computed using the geometric method. Explicitly, we prove the following result:

\begin{thm} \label{thm:main}
Let $m,n$ be natural numbers with $\gcd(m,n)=1$. Let $\bk$ be any field with characteristic $\charF(\bk)$ not dividing $n$ and $m$ or $\charF(\bk)=0$.
Denote by $\xi_{l}^{\bk}$ the number of $l$-th roots of unity in $\bk$. Then, the motive of the $\AGL{1}$-representation 
variety of the $(m,n)$-torus knot in the Grothendieck ring $\K{\Var{\bk}}$ of $\bk$-algebraic varieties is:
\begin{align*}
	\big[\Rep{m,n}&(\AGL{1}) \big] = ( \xi_{nm}^{\bk}-\xi_{n}^{\bk}-\xi_{m}^{\bk}+2)(q^2-q),
\end{align*}
where $q = [\bk] \in \K{\Var{\bk}}$ is the Lefschetz motive.

Suppose in addition that $n,m$ are both odd and $\charF(\bk) \neq 2$. Then the motive of the $\AGL{2}$-representation variety of the $(m,n)$-torus knot is:
\begin{align*}
	\big[\Rep{m,n}(\AGL{2})\big] =& 
	\;\frac14 \left(q^7 + q^6 + q^5 - 5q^4 + 2q^3\right) + \frac14 \left(\xi^\bk_m\xi^\bk_n - \xi^\bk_m - \xi^\bk_n\right)(q - 1)^2(q - 2)(q+1)q^3  \\
     & \;+ \frac14(\xi^{\bk}_m-1)(\xi^{\bk}_n-1)(q^5-q^3)\bigg(2\left(\xi^{\bk}_m\xi^{\bk}_n-\xi^{\bk}_m-\xi^{\bk}_n + 2\right)(q - 1) \\
     & \;\;\;\;\;\;\;\;\;\;\;\;+ (q - 1)(q - 2)\left(\left(\xi^{\bk}_m-2)(\xi^{\bk}_n-2\right)q+\xi^{\bk}_m\xi^{\bk}_n-4\right)\bigg).
\end{align*}
\end{thm}

As a by-product of the computation of this motive, just by seeing $q$ as a formal variable, we directly obtain the number of points of $\Rep{m,n}(\AGL{r})$ over a finite field $\bk = \FF_q$ as well as its $E$-polynomial over $\bk = \CC$ (where we have $\xi_l^{\CC} = l$ and the result agrees with \cite{proceedings-Rassias}). In some sense, this result jeopardizes the arithmetic method. Indeed, the number of $l$-th roots of unity $\xi_l^{\FF_q}$ in a finite field $\FF_q$ strongly depends on the divisibility of $l$ and $q$. For this reason, Theorem \ref{thm:main} shows that, instead of a single counting polynomial that recovers the number of points on finite fields, several counting polynomial arise depending on the arithmetic between $n,m$ and $q$, reflecting the fact that several trends in the growth of the number of points co-exist.

Among these trends, only one of them is `the right one' in the sense that its counting polynomial agrees with the $E$-polynomial. Nevertheless, unraveling it may be a difficult task. Indeed, as we shall show in Section \ref{sec:counting}, this `right trend' is a minority: in terms of density, for most of the prime powers $q$ the number of points of the representation variety over $\FF_q$ lies in a wrong trend. This poses an epistemological problem: unless you know the $E$-polynomial beforehand, you cannot detect the right trend that computes it.

Furthermore, Theorem \ref{thm:main} also evidences that the gap between polynomial counting and motivic theory is bigger than expected. Given a complex variety $X$, it may happen that $X$ lies in the subring of $\K{\Var{\CC}}$ generated by the Lefschetz motive, that is, $[X] = P(q)$ for some polynomial $P(t) \in \ZZ[t]$. In that case, we shall say that $X$ is \emph{polynomial motivic} and $P$ will be its \emph{motivic polynomial} (for further information, refer to Section \ref{sec:counting}).
This is the natural analogue of the polynomial count property in the motivic framework. In some sense, the motivic polynomial is reflecting some kind of transcendental-arithmetic property that instantiates as the number of points over finite fields and as the motive over $\CC$. For this reason, it makes sense to pose the following conjecture which, in some form or another, has been implicitly considered in many previous works on representation varieties.

\begin{conjecture}\label{conj:naive}
A complex variety is asymptotically polynomial count if and only if it is polynomial motivic, and the counting and motivic polynomials agree.
\end{conjecture}

However, despite that Conjecture \ref{conj:naive} is valid for all the previously known examples of representation varieties of surface groups, it turns out that it does not hold true for the varieties of Theorem \ref{thm:main}. In fact, in Section \ref{sec:counting} we prove the following corollary of Theorem \ref{thm:main}.

\begin{cor}\label{cor:motivic-not-asympt}
The $\AGL{1}$-representation variety of torus knots is polynomial motivic but not asymptotically polynomial count.
\end{cor}

As we will discuss in Section \ref{sec:counting}, 
the previous result does not fully prevent Katz theorem to work in this context. A slightly stronger version of this result is proven in \cite{Hausel-Rodriguez-Villegas:2008} that allows us to discard the wrong countings. The idea is the following. Suppose that we manage to find a ring $A$ such that, for any embedding $A \hookrightarrow \FF_q$ into a finite field, the variety $X(\FF_q)$ does admit polynomial count (in the sense that there exists a polynomial $P(t) \in \ZZ[t]$ such that $P(q) = |X(\FF_q)|$). The choice of the initial ring $A$ allows us to get rid of the `bad prime powers' so the `right trend' is precisely those prime powers $q$ that admit an embedding $A \hookrightarrow \FF_q$. For this reason, this property is called \emph{strongly polynomial count}.

In our case, it turns out that taking $A = \ZZ[\xi_n, \xi_m]$, where $\xi_n, \xi_m \in \CC$ are primitive $n$-th and $m$-th roots of unity, we are able to
dismiss all the finite fields lying in a `wrong trend' and to select precisely the prime powers whose counting polynomial agrees with the $E$-polynomial. In this sense, we shall show in Section \ref{sec:counting} the following.

\begin{cor}\label{cor:strongly}
The $\AGL{1}$-representation variety of torus knots is strongly polynomial count.
\end{cor}

The combination of Corollaries \ref{cor:motivic-not-asympt} and \ref{cor:strongly} provides a counterexample to Conjecture \ref{conj:naive}. For this reason, instead of the na\"ive Conjecture \ref{conj:naive}, it makes sense to pose the following subtler conjecture which, to the best of our knowledge, remains open.

\begin{conjecture}
A complex variety is strongly polynomial count if and only if it is polynomial motivic, and the counting and motivic polynomials agree.
\end{conjecture}

\noindent \textbf{Acknowledgements.} The authors thank Marton Hablicsek and Jesse Vogel for very useful conversations regarding the Grothendieck ring of algebraic varieties. The first author was partially supported by MICINN (Spain) grant PID2019-106493RB-I00.
The second author was partially supported by MINECO (Spain) grant PGC2018-095448-B-I00.

\section{Basic notions} 

\subsection{Representation varieties of torus knots}\label{sec:rep-var}

Let $\bk$ be any field, which in general is not assumed to be algebraically closed. 
Let $\G$ be a finitely presented group, and let $G$ be an algebraic group over $\bk$. 
A \textit{representation} of $\G$ in $G$ is a homomorphism $\rho: \G\to G$.
Consider a presentation $\G=\langle x_1,\ldots, x_k \,|\, r_1,\ldots, r_s \rangle$. Then $\rho$ is completely
determined by the $k$-tuple $(A_1,\ldots, A_k)=(\rho(x_1),\ldots, \rho(x_k))$
subject to the relations $r_j(A_1,\ldots, A_k)=\Id$, $1\leq j \leq s$. The variety 
of representations is
 \begin{eqnarray*}
 \Rep{\G}(G) &=& \Hom(\G, G) \\
  &=& \{(A_1,\ldots, A_k) \in G^k \, | \,
 r_j(A_1,\ldots, A_k)=\Id, \,  1\leq j \leq s \}\subset G^{k}\, .
 \end{eqnarray*}
Written in this way, $\Rep{\G}(G)$ is an affine algebraic set and it can be checked that this algebraic structure does not depend on the chosen presentation.

If $G< \GL{d}$, then a representation $\rho$ gives an endomorphism of the $\bk$-vector space 
$\bk^d$, so it makes sense to talk about the reducibility of $\rho$. A representation $\rho$ is \textit{reducible} if there exists some proper 
subspace $V\subset \bk^d$ such that for  all $g\in G$ we have 
$\rho(g)(V)\subset V$;  otherwise $\rho$ is
\textit{irreducible}. 

We will be interested in representations of the fundamental group of the complement of a torus knot. To be precise, let $T^2=S^1 \times S^1$ be the $2$-torus and consider the standard embedding
$T^2\subset S^3$. Let $m,n$ be a pair of coprime positive integers. Identifying
$T^2$ with the quotient $\RR^2/\ZZ^2$, the image of the straight line $y=\frac{m}{n}
x$ in $T^2$ defines the \textit{torus knot} of type $(m,n)$, which we shall denote
as $K_{m,n}\subset S^3$ (see \cite[Chapter 3]{Ro}).
Now, if we take $\Gamma_{m,n} = \pi_1(S^3-K_{m,n})$, it is well-known that a presentation of this group is given by
 $$
  \G_{m,n} = \langle x,y \, | \, x^n= y^m \,\rangle \,.
 $$
Therefore, we can consider the variety of representations of the torus knot of type $(m,n)$, which is explicitly described as
 $$
  \Rep{m,n}(G)=\Rep{\G_{m,n}}(G)=\{ (A,B) \in G^2 \,|\, A^n=B^m\}. 
  $$

\subsection{The Grothendieck ring of algebraic varieties}

Let us consider the category $\Var{\bk}$ of algebraic varieties with regular morphisms 
over a base field $\bk$. The isomorphism classes of algebraic varieties (i.e.\ of objects of $\Var{\bk}$) form a semi-ring with the disjoint union as addition and the cartesian product as multiplication. This semi-ring can be promoted to a ring, the so called {\em Grothendieck group of algebraic varieties} $\K{\Var{\bk}}$, by adding formal additive inverses. Explicitly, it is the abelian group generated by isomorphism classes of algebraic varieties with the relation that $[X]=[Y]+[U]$ if $X=Y\sqcup U$, with $Y\subset X$ a closed subvariety. The image of an algebraic variety $X$ in $\K{\Var{\bk}}$, $[X]$, is usually referred to as the {\em virtual class of $X$} or its {\em motive}. A very important element of this ring is the class of the affine line, $q = [\bk] \in \K{\Var{\bk}}$, the so-called \emph{Lefschetz motive}. 

\begin{rmk} \label{rem:zerodivisor}
Despite the simplicity of its definition, the ring structure of $\K{\Var{\bk}}$ is widely unknown. In particular, for almost fifty years it was an open problem whether it is an integral domain. The answer is no and, more strikingly, the Lefschetz motive $q$ is a zero divisor \cite{Borisov:2014}.
\end{rmk}

Virtual classes are well-behaved with respect to two typical geometric situations that we will encounter in the upcoming sections. A proof of the
following facts can be found for instance in \cite[Section 4.1]{GP:2018b}{}. Let $\bk$ be a field of characteristic $\charF(\bk)\neq 2$. 
\begin{itemize}
	\item Let $E \to B$ be a regular morphism that is a locally trivial bundle in the Zariski topology with fiber $F$. In this situation, we have
	that in $\K\Var{\bk}$
	$$
		[E]=[F]\cdot [B].
	$$
	\item Suppose that $X$ is an algebraic variety with an action of $\ZZ_2$.
Setting $[X]^+ = [X / \ZZ_2]$ and $[X]^- = [X] - [X]^+$, we have the formula
  \begin{equation}\label{eqn:2bis}
 [X \times Y]^{+} = [X]^+[Y]^+ + [X]^-[Y]^- 
  \end{equation}
for two varieties $X,Y$ with $\ZZ_2$-actions. 
\end{itemize}

\begin{ex}\label{ex:calculations-kvar}
Consider the fibration $\bk^2 - \bk \to \GL{2} \to \bk^{2}-\{(0,0)\}$, $f \mapsto f(1,0)$. It is locally trivial in the Zariski topology,
and therefore $[\GL{2}]=[\bk^2-\bk]\cdot[\bk^{2}-\{(0,0)\}] = (q^2-q)(q^2-1) = q^4-q^3-q^2+q$. Analogously, the quotient map defines a
locally trivial fibration $\bk^* = \bk - \{0\} \to \GL{2} \to \PGL{2}$, so $[\PGL{2}] = q^3-q$.
\end{ex}

The following is proved in \cite[Lemma 2.2]{proceedings-Rassias} for $\bk=\CC$, but the proof extends to any field verbatim.

\begin{lem}\label{lem;ac}
Let $\ZZ_2$ act on $\bk^2$ by exchange of coordinates. Then 
$[(\bk^*)^2 - \Delta]^+ = (q-1)^2$, $[(\bk^*)^2 - \Delta]^- = -q+1$, where $\Delta$ denotes the diagonal.

Also let $X=\GL{2} / \GL{1} \times \GL{1}$, and $\ZZ_2$ acting by exchange of columns in $\GL{2}$.
Then $[X]^+=q^2$ and $[X]^-=q$.
\end{lem}

There is a subtle point in the definition of $\K{\Var{\bk}}$ when dealing with non-algebraically closed fields $\bk$ due to the failure of the Nullstellensatz. Recall that, with the modern definition, an affine variety is a scheme of the form $X= \Spec(R)$ with $R$ a reduced $\bk$-algebra of finite type which is also an integral domain, and the closed points of $X$ are precisely the maximal ideals of $R$. If we write $R = \bk[x_1, \ldots, x_n]/I$ with $I$ an ideal, when $\bk$ is algebraically closed then these maximal ideals of $R$ are in one-to-one correspondence with the common zeros in $\bk^n$ of the elements of $I$. This recovers the classical definition of an affine variety as the zero set of polynomials. However, if $\bk$ is not algebraically closed then this correspondence is no longer bijective. For instance, the scheme $X = \Spec\left(\RR[x]/(x^2 + 1)\right)$ defines a non-trivial affine variety, but the zero set of $I = (x^2+1)$ in $\RR$ is empty.

This opens the door to two different definitions for an affine variety over $\bk$: as the spectrum of a finitely generated integral domain over $\bk$ (the scheme-theoretic one) or as the zero set of an polynomial ideal (the classical one). Furthermore, we can define an (abstract) algebraic variety as a separated integral scheme of finite type over $\bk$ or as a locally ringed space locally isomorphic to a classical affine variety. Let us denote by $\Var{\bk}^{sc}$ the category of the former and by $\Var{k}^{cl}$ that of the later. Their associated Grothendieck rings are related by a natural map
$$
	\K{\Var{\bk}^{sc}} \to \K{\Var{\bk}^{cl}},
$$
given by taking the `underlying zero set' of the scheme. When $\bk$ is not algebraically closed, this map may not be a monomorphism and its kernel is the collection of scheme-theoretic algebraic varieties with empty zero-set trace.

Throughout this paper, we shall work with the classical definition, which is better suited for geometric arguments. In this way, when we write $\Var{\bk}$ we will always refer to  $\Var{\bk}^{cl}$.

\subsection{Hodge structures}\label{sec:Hodge}

In the complex case, $\bk = \CC$, our algebraic varieties are naturally endowed with an extra structure, the so-called Hodge structure.

Recall that a pure Hodge structure of weight $k$ consists of a finite dimensional complex vector space
$H$ with a real structure, and a decomposition $H=\bigoplus_{k=p+q} H^{p,q}$
such that $H^{q,p}=\overline{H^{p,q}}$, the bar meaning complex conjugation on $H$.
A pure Hodge structure of weight $k$ gives rise to the so-called Hodge filtration, which is a descending filtration $F^{p}=\bigoplus_{s\ge p}H^{s,k-s}$. We define $\Gr^{p}_{F}(H):=F^{p}/ F^{p+1}=H^{p,k-p}$.
 More generally, a mixed Hodge structure consists of a finite dimensional complex vector space $H$ with a real structure,
an ascending (weight) filtration $W_\bullet$ of $H$ defined over $\RR$, and a descending (Hodge) filtration $F^\bullet$ such that $F^\bullet$ induces a pure Hodge structure of weight $k$ on each graded piece $\Gr^{W}_{k}(H)=W_{k}/W_{k-1}$. Set $H^{p,q} = \Gr^{p}_{F}\Gr^{W}_{p+q}(H)$ and write $h^{p,q}$ for the dimension $h^{p,q} :=\dim H^{p,q}$, usually referred to as the {\em Hodge numbers}.

Now, let $X$ be any quasi-projective complex algebraic variety (possibly non-smooth or non-compact). 
The real cohomology groups $H^k(X)$ and the cohomology groups with compact support  
$H^k_c(X)$ are endowed with mixed Hodge structures \cite{De}{}, and the complex algebraic maps preserve them. 
We define the {\em Hodge numbers} of $X$ by
$h^{k,p,q}_{c}(X)= h^{p,q}(H_{c}^k(X))=\dim \Gr^{p}_{F}\Gr^{W}_{p+q}H^{k}_{c}(X)$.

\begin{defn}
Given a complex quasi-projective algebraic variety, the {\em $E$-polynomial}, also known as the Hodge-Deligne polynomial, is defined as 
 $$
 e(X)=e(X)(u,v):=\sum _{p,q,k} (-1)^{k}h^{k,p,q}_{c}(X)\, u^{p}v^{q} \in \ZZ[u,v].
 $$
\end{defn}

\begin{rmk}
When $h_c^{k,p,q}(X)=0$ for $p\neq q$, the polynomial $e(X)$ depends only on the product $uv$.
This will happen in all the cases that we shall investigate here. In this situation, it is
conventional to use the variable $q=uv$. For instance, $e(\CC^n)=q^n$.
\end{rmk}

A key property of $E$-polynomials that permits their calculation is that they are additive for
stratifications of $X$. If $X$ is a complex algebraic variety and we decompose it into $X=\bigsqcup_{i=1}^{n}X_{i}$, where all the $X_i$ are locally closed in $X$, then 
 $$
 e(X)=\sum_{i=1}^{n}e(X_{i}).
 $$
The $E$-polynomial is also a multiplicative mapping. Indeed, the K\"unneth isomorphism shows that $e(X \times Y) = e(X)e(Y)$. Additionally, this multiplicative property can be extended to more general scenarios, such as algebraic fibrations that are locally trivial in the Zariski topology, or principal $G$-bundle with $G$ a connected algebraic group \cite[Remark 2.5]{LMN}.

Due to these additivity and multiplicativity properties, the $E$-polynomial defines a ring homomorphism
$$
	e: \K{\Var{\CC}} \to \ZZ[u^{\pm 1}, v^{\pm 1}].
$$
This homomorphism factorizes through mixed Hodge structures. To be precise, the category of mixed Hodge structures ${\MHS{\QQ}}$ is an abelian category \cite{De}{}. Therefore we may as well consider its Grothendieck group, $\K{\MHS{\QQ}}$, which again inherits a ring structure. The long exact sequence in cohomology with compact support and the K\"unneth isomorphism shows that there exist ring homomorphisms $\K{\Var{\CC}} \to \K{\MHS{\QQ}}$ given by $[X] \mapsto \sum_k (-1)^k[H_c^k(X)]$ and $\K{\MHS{\QQ}} \to \ZZ[u^{\pm 1}, v^{\pm 1}]$ given by $[H] \mapsto \sum h^{p,q}(H) u^{p}v^{q}$ such that the following diagram commutes
	\[
\begin{displaystyle}
   \xymatrix
   {	\K{\Var{\CC}} \ar[rd]_{e} \ar[r] & \K{\MHS{\QQ}} \ar[d] \\
   & \ZZ[u^{\pm 1}, v^{\pm 1}]
      }
\end{displaystyle}   
\]

\begin{rmk}\label{remark:$E$-pol-lefschetz}
From the previous diagram, we get that the $E$-polynomial of the affine line is $q=e([\CC])$ which justifies denoting by $q=[\CC]\in\K{\Var{\CC}}$ the Lefschetz motive. This implies that if the motive of a variety lies in the subring of $\K{\Var{\CC}}$ generated by the affine line, then the $E$-polynomial of the variety coincides with the motive, by seeing $q$ as a variable. 
\end{rmk}

We can also consider the equivariant version of the $E$-polynomial. Let $X$ be a complex quasi-projective variety on which a finite group $F$ acts. 
Then $F$ also acts of the cohomology $H^k_c(X)$ respecting the mixed Hodge structure, so
$[H^k_c(X)]$ can be seen as an element of the representation ring $R(F)$ of $F$. The \emph{equivariant 
Hodge-Deligne polynomial} is defined as
 $$
 e_F(X)=\sum_{p,q,k}  (-1)^k [H^{k,p,q}_c(X)] \, u^pv^q \in R(F)[u,v].
 $$

For instance, for an action of $\ZZ_2$, there are two irreducible
representations $T,N: \ZZ_2 \to \CC^*$, where $T$ is the trivial representation, and $N$ is the non-trivial representation.
Then, the equivariant $E$-polynomial can be written as $e_{\ZZ_2}(X)=aT+bN$, where $e(X) = a+b$, $e(X/\ZZ_2) = a$, thus $b=e(X)-e(X/\ZZ_2)$.
In the notation of \cite[Section 2]{LMN}, $a=e(X)^+$, $b=e(X)^-$. 
Note that if $X,X'$ are complex varieties with $\ZZ_2$-actions, then 
writing $e_{\ZZ_2}(X)=aT+bN$ and $e_{\ZZ_2}(X')=a'T+b'N$, we have
$e_{\ZZ_2}(X\x X')= (aa'+bb') T+ (ab'+ba')N$ and so
  \begin{equation}\label{eqn:2}
  e((X\x X')/\ZZ_2)= e(X)^+e(X')^++ e(X)^-e(X')^-.
 \end{equation}
This is the Hodge-theoretic analogue of (\ref{eqn:2bis}).

%
%

%

\section{The problem of polynomial counting} \label{sec:counting}

In the seminal paper \cite{Hausel-Rodriguez-Villegas:2008}, Hausel and Rodr\'iguez-Villegas introduced an arithmetic method for computing the $E$-polynomial of representation varieties. The key idea was to count the number of points of a complex variety over finite fields, as suggested by the Weil conjectures.

\begin{defn}
Let $X$ be a complex algebraic variety. A \emph{spreading out} of $X$ is a pair $(A, \cX)$ of a finitely generated ring $A \subset \CC$ and a separated $A$-scheme $\cX$ such that $X$ and $\cX(\CC) = \cX \times_{\textrm{Spec}(A)} \textrm{Spec}(\CC)$ are isomorphic as complex algebraic varieties.
\end{defn}

Suppose that $A$ is also a finitely generated $\bk$-algebra. For any morphism $A \to A'$ to a finitely generated $\bk$-algebra, we will denote by $\cX(A')$ the set of closed points of $\cX \times_{\textrm{Spec}(A)} \textrm{Spec}(A')$ after the extension of scalars. In particular, if $A'=\FF_q$ is a finite field, then $\cX(\FF_q)$ is a finite set of cardinal $|\cX(\FF_q)|$.

\begin{defn}
Let $X$ be a complex algebraic variety. We say that $X$ is \emph{asymptotically polynomial count} if there exist a spreading out $(\ZZ, \cX)$ of $X$ and a polynomial $P(t) \in \ZZ[t]$ such that $P(p^n) = \left|\cX(\FF_{p^n})\right|$ for all but finitely many prime $p$ and all $n \geq 1$.
\end{defn}

A related notion is the following, where the key difference is that, now, the ground ring $A$ of the spreading out may be any finitely generated subring of $\CC$.

\begin{defn}
A complex variety $X$ is said to be \emph{strongly polynomial count} if there exist a spreading out $(A, \cX)$ of $X$ and a polynomial $P(t) \in \ZZ[t]$ such that, for any monomorphism $A \hookrightarrow \FF_q$ we have that $P(q) = \left|\cX(\FF_q)\right|$. 
\end{defn}

\begin{rmk}\label{rmk:asymp-strong}
An asymptotically polynomial count complex variety is strongly polynomial count. Indeed, suppose that we have a spreading out $(\ZZ, \cX)$ such that $P(t)$ counts the number of points of points in $\cX(\FF_q)$ for $q > N_0$. Let $p > N_0$ be a prime number and take $A = \ZZ[\xi_p]$ where $\xi_p$ is a $p$-th root of unit in $\CC$. Then any monomorphism $A \hookrightarrow \FF_q$ must send $\xi_p$ to an non-vanishing element of order $p$ which forces that $q \geq p$. Therefore, $P(q) = |\cX(\FF_q)|$ and thus $X$ is strongly polynomial count.
\end{rmk}

The very remarkable feature of these varieties is that this counting process actually captures their Hodge structure.

\begin{thm}[{\cite[Theorem 6.1.2]{Hausel-Rodriguez-Villegas:2008}}]\label{thm:hausel-counting}
If $X$ is a complex variety that is strongly polynomial count with counting polynomial $P$, then the $E$-polynomial of $X$ is
$$
	e(X)(u,v) = P(uv).
$$
\end{thm}

\begin{ex}
The affine space $\CC^n$ is polynomial counting. It admits a universal spreading out $(\ZZ, \ZZ^n)$ and the polynomial $P(t)=t^n$ is its counting polynomial since $|\ZZ^n(\FF_q)| = |\FF_q^n| = q^n$. This is in perfect agreement with the fact that its $E$-polynomial must be $e(\CC^n)(u,v) = u^nv^n = q^n$, as it is provided by the relation $1 + q + q^2 + \ldots + q^n = e(\PP_\CC^n) = e(\CC^n) + e(\PP_\CC^{n-1})$.

On the other hand, a Riemann surface $\Sigma$ of genus $g \geq 1$ is not polynomial count. Recall that its (pure) Hodge structure is $H^{0,0}(\Sigma) = \CC$, $H^{1,0}(\Sigma) = H^{0,1}(E) = \CC^g$ and $H^{1,1}(\Sigma) = \CC$. Hence, its $E$-polynomial is given by $e(\Sigma)(u,v) = 1 - gu-gv + uv$. Notice that $e(\Sigma)(u,v)$ is not even a polynomial in the variable $uv$, so $\Sigma$ cannot be polynomial count.
\end{ex}

A related property to polynomial count can be read in the Grothendieck ring of algebraic varieties, $\K{\Var{\CC}}$.

\begin{defn}
A complex algebraic variety $X$ will be said to be \emph{polynomial motivic} if its motive $[X] \in \K{\Var{\CC}}$ can be expressed as $[X] = P(q)$ for some polynomial $P \in \ZZ[t]$, that we will refer to as the \emph{motivic polynomial}. In other words, the motive $[X]$ belongs to the subring of $\K{\Var{\CC}}$ generated by the Lefschetz motive $q$. 
\end{defn}

The interplay between polynomial count and polynomial motivic varieties is widely presented in the literature about representation varieties. 
In \cite{Hausel-Rodriguez-Villegas:2008}, the authors showed that the twisted $\GLC{r}$-representation varieties of Riemann surfaces are polynomial count and computed their counting polynomials. This work on polynomial count has been extended in \cite{mereb} for $\SLC{r}$ and in \cite{Hausel-Letelier-Villegas} for parabolic representation varieties. At the other side, in the paper \cite{MM} it was shown that the $\SLC{2}$-representation varieties of Riemann surfaces are also polynomial motivic and the motivic polynomial agrees with the counting polynomial. Additionally, the recent developments on Topological Quantum Field Theories has opened the door to show that general representation varieties are also polynomial motivic and to compute their motives \cite{GPLM-2017,GP:2018a}.

Based on these results, and the similarity between these concepts, the following conjecture has been considered in one or another form in the aforementioned works.

\begin{conjecture}\label{conj:motivic}
Let $\cX$ be a reduced $\ZZ$-scheme and let $X = \cX(\CC)$ be the associated complex algebraic variety by extension of scalars. Then $X$ is asymptotically polynomial count if and only if it is polynomial motivic, and the counting and motivic polynomials agree.
\end{conjecture}

We shall disprove this conjecture with our computations of $\AGL{r}$-representation varieties of torus knots. Here, $\AGL{r}$ is the group of affine transformations of the $r$-dimensional affine space. In other words, the elements of $\AGL{r}$ are matrices of the form
$$
\begin{pmatrix}
	1 & 0 \\
	\alpha & A_0
	\end{pmatrix},
$$
with $\alpha \in \bk^r$ and $A_0 \in \GL{r}$. Multiplication in $\AGL{r}$ is given by matrix multiplication, so we have a natural description as semi-direct product
$\AGL{r} = \bk^{r} \rtimes \GL{r}$.

\section{$\AGL{1}$-representation varieties of torus knots}\label{sec:AGL1}

In this section, we compute the motive of the $\AGL{1}$-representation variety of the $(m,n)$-torus knot by describing it explicitly. Fix a field $\bk$ whose characteristic is $\charF(\bk)=0$ or $\charF(\bk) > 0$ not dividing both $n$ and $m$. 
Suppose that we have an element $(A, B) \in \Rep{m,n}(\AGL{1})$ with matrices of the form
$$
	A = \begin{pmatrix}
	1 & 0 \\
	\alpha & a_0
	\end{pmatrix}, \qquad B = \begin{pmatrix}
	1 & 0 \\
	\beta & b_0
	\end{pmatrix}.
$$
A straightforward computation shows that
$$
	A^n = \begin{pmatrix}
	1 & 0 \\
	(1 + a_0 + \ldots +a^{n-1}_0)\alpha & a^n_0
	\end{pmatrix}, \qquad B^m = \begin{pmatrix}
	1 & 0 \\
	(1 + b_0 + \ldots+ b^{m-1}_0)\beta & b^m_0
	\end{pmatrix}.
$$

\begin{lem}\label{lem:iso-cusp}
Let $m,n \geq 1$ be coprime natural numbers. Then the algebraic curve
$$
	C = \left\{(x,y) \in \bk^2 - \left\{(0,0)\right\}\,|\, x^n = y^m \right\}
$$
is isomorphic to $\bk^* = \bk - \left\{0\right\}$ under the map $t \in \bk^* \mapsto (t^m, t^n) \in C$.

\begin{proof}
Let $a,b$ be integers such that $am + bn = 1$. Then, the inverse of the claimed map is $(x, y) \in C \mapsto x^ay^b \in \bk^*$.
\end{proof}
\end{lem}

\begin{rmk}
Since $a < 0$ or $b<0$, the above-mentioned inverse map $(x, y) \mapsto x^a y^b$ is not defined for $(x,y)=(0,0)$. Of course, this agrees with the fact that $C \cup \left\{(0,0)\right\}$ is not a smooth curve.
\end{rmk}

Under the isomorphism of Lemma \ref{lem:iso-cusp}, the representation variety can be explicitly described as
$$
\Rep{m,n}(\AGL{1}) = \left\{(t, \alpha, \beta) \in \bk^* \times \bk^2 \,\left|\, \Phi_n(t^m)\alpha = \Phi_m(t^n)\beta\right.\right\},
$$
where, for $l \geq 1$, $\Phi_l$ is the polynomial $\Phi_l(x) = 1 + x + \ldots + x^{l-1} \in \bk[x]$. Written in a more geometric fashion, the morphism $(t, \alpha, \beta) \mapsto t$ defines a regular map
\begin{align}\label{eq:fibration-agl1}
	\Rep{m,n}(\AGL{1}) \stackrel{}{\longrightarrow} \bk^*.
\end{align}
The fiber over $t \in \bk^*$ is the orthogonal complement of the vector $(\Phi_n(t^m), -\Phi_m(t^n)) \in \bk^2$. 
This complement is $\bk$ if $(\Phi_n(t^m), \Phi_m(t^n)) \neq (0,0)$ and is $\bk^2$ otherwise.

Recall that if $\charF(\bk)$ does not divide $l$, then the roots of $\Phi_l$ are the $l$-th roots of unit different from $1$. Denote by $\mu_l^{\bk}$ the group of $l$-th roots of units in $\bk$ (including $1$). In our case, the assumptions on the characteristic of $\bk$ imply that $(\Phi_n(t^m), \Phi_m(t^n)) = (0,0)$ if and only if $t \in \mu_{nm}^{\bk} - \left(\mu_n^{\bk} \cup \mu_m^{\bk}\right)$. Set 
$$
\Omega_{m,n}^\bk = \mu_{nm}^{\bk} - \left(\mu_n^{\bk} \cup \mu_m^{\bk}\right),
$$ 
and notice that it is a finite set. To count its elements, set $\xi_l^{\bk} = |\mu_l^{\bk}|$. Since $m,n$ are coprime $\mu_n^\bk \cap \mu_m^\bk = \left\{1\right\}$ and thus
$$
|\Omega_{m,n}^{\bk}| = \xi_{nm}^{\bk}-\xi_{n}^{\bk}-\xi_{m}^{\bk}+1.
$$ 
Note that if $\bk$ is algebraically closed, then $\xi_l^{\bk} = l$ and, thus, $|\Omega_{m,n}^\bk| = nm-n-m+1 = (n-1)(m-1)$.

Therefore, (\ref{eq:fibration-agl1}) decomposes into the two Zariski locally trivial fibrations
 \begin{align*}
	&\bk \longrightarrow \Rep{m,n}^{(1)}(\AGL{1}) \longrightarrow \bk^* - \Omega_{m,n}^\bk, \\ 
	&\bk^2 \longrightarrow \Rep{m,n}^{(2)}(\AGL{1}) \longrightarrow \Omega_{m,n}^\bk,
 \end{align*}
with $\Rep{m,n}(\AGL{1})  = \Rep{m,n}^{(1)}(\AGL{1}) \sqcup \Rep{m,n}^{(2)}(\AGL{1})$. Thus, this implies that the motive of the whole representation variety is
\begin{align*}
	\left[\Rep{m,n}(\AGL{1}) \right] &= \left[\Rep{m,n}^{(1)}(\AGL{1}) \right] + \left[\Rep{m,n}^{(2)}(\AGL{1}) \right]  = \left[\bk^* - \Omega_{m,n}^{\bk}\right]\left[\bk\right] + \left[\Omega_{m,n}^{\bk}\right]\left[\bk^2\right] \\
	& = (q-1 - |\Omega_{m,n}^{\bk}|)q + |\Omega_{m,n}^{\bk}|q^2 = ( \xi_{nm}^{\bk}-\xi_{n}^{\bk}-\xi_{m}^{\bk}+2)(q^2-q).
\end{align*}

\section{Counterexample to Conjecture \ref{conj:motivic}}

In this section, we give a counterexample to Conjecture \ref{conj:motivic}. In fact, the counterexample will be the 
$\AGL{1}$-representation variety of a $(m,n)$-torus knot with $n$ and $m$ coprime numbers. The core of this fact is the following easy 
lemma, whose proof is provided for completeness.

\begin{lem}
The number of $l$-th roots of unit in $\FF_q$ is $\xi_l^{\FF_q}=\gcd(l, q-1)$.
\end{lem}

\begin{proof}
Recall that $\FF_q^* = \FF_q - \{0\}$ is a cyclic group of order $q-1$. Under the isomorphism $\FF_q^* \cong \ZZ_{q-1}$ the $l$-th roots of unit correspond to the annihilators of $l$, that is, the elements $a \in \ZZ_{q-1}$ such that $al \equiv 0 \pmod{q-1}$. There are exactly $\gcd(l,q-1)$ of these.
\end{proof}

\begin{rmk}
The number of roots of $\Phi_l(x)$ in $\FF_q$ is $\xi_l^{\FF_q}-1$, regardless of whether they are repeated. This should be necessarily the case when $l$ does not divide $q$. 
\end{rmk}

\begin{cor}\label{cor:counterexample-asymp-count}
Let $m,n > 2$ be natural numbers with $\gcd(m,n)=1$. Then the representation variety $\Rep{m,n}(\AGLC{1})$ is not asymptotically polynomial count.
\end{cor}

\begin{proof}
In Section \ref{sec:AGL1} we proved that the motive of $\Rep{m,n}(\AGL{1})$ for any field $\bk$ is
\begin{align*}
	\left[\Rep{m,n}(\AGL{1}) \right] &= ( \xi_{nm}^{\bk}-\xi_{n}^{\bk}-\xi_{m}^{\bk} + 2)(q^2-q),
\end{align*}
where $\xi_l^{\bk}$ is the number of $l$-th roots of unit in $\bk$. In particular, this shows that on 
a finite field $|\Rep{m,n}(\mathrm{AGL}_{1}(\FF_q))| = ( \xi_{nm}^{\FF_q}-\xi_{n}^{\FF_q}-\xi_{m}^{\FF_q}+1)(q^2-q)$.

By Dirichlet theorem on arithmetic progressions, there exist an infinite sequence of primes $\left\{p_\alpha\right\}_{\alpha=1}^\infty$ satisfying $p_\alpha \equiv 1 \pmod{nm}$ which, under the isomorphism $\ZZ_{nm} \cong \ZZ_n \times \ZZ_m$, correspond to $p_\alpha \equiv 1 \pmod{n}$ 
and $p_\alpha \equiv 1 \pmod{m}$. In other words, $n$ and $m$ divide $p_\alpha-1$ so $\gcd(n,p_\alpha-1) = n$, $\gcd(m, p_\alpha-1) = m$ and $\gcd(nm, p_\alpha-1) = nm$. Therefore, for these primes we have that $\xi_n^{\FF_{p_\alpha}} = n$, $\xi_m^{\FF_{p_\alpha}} = m$ and $\xi_{nm}^{\FF_{p_\alpha}} = nm$ so $|\Rep{m,n}(\mathrm{AGL}_{1}(\FF_{p_\alpha}))| = ((n-1)(m-1)+1)(p_\alpha^2-p_\alpha)$. This shows that if $\Rep{m,n}(\AGLC{1})$ is asymptotically polynomial count, its counting polynomial must be $P(t) = (nm-n-m+2)(t^2-t)$.

However, applying again Dirichlet theorem we also find an infinite sequence of primes $\left\{q_\beta\right\}_{\beta=1}^\infty$ satisfying $q_\beta \equiv 2 \pmod{nm}$ or, equivalently, $q_\beta \equiv 2 \pmod{n}$ and $q_\beta \equiv 2 \pmod{m}$. In this case, this implies that $\xi^{\FF_{q_\beta}}_n = \gcd(n,q_\beta-1) < n$, $\xi^{\FF_{q_\beta}}_m = \gcd(m, q_\beta-1) < m$ and $\xi^{\FF_{q_\beta}}_{nm} = \xi^{\FF_{q_\beta}}_n \xi^{\FF_{q_\beta}}_m = \gcd(nm, q_\beta-1) < nm$. Therefore, $|\Rep{m,n}(\mathrm{AGL}_{1}(\FF_{q_\beta}))| = ((\xi^{\FF_{q_\beta}}_n-1)(\xi^{\FF_{q_\beta}}_m-1)+1)(q_\beta^2-q_\beta) < P(q_\beta)$, contradicting that $P(t)$ is the counting polynomial.
\end{proof}

Due to Corollary \ref{cor:counterexample-asymp-count}, we see that Conjecture \ref{conj:motivic} does not hold true. 
Nevertheless, on the contrary we have the following result.

\begin{prop}
For any $m,n > 0$ with $\gcd(m,n)=1$, the variety $\Rep{m,n}(\AGLC{1})$ is strongly polynomial count.
\end{prop}

\begin{proof}
The idea is the same as in Remark \ref{rmk:asymp-strong}. Pick elements $\xi_n, \xi_m \in \CC$ 
which are primitive $n$-th and $m$-th roots of unit and set $A = \ZZ[\xi_n, \xi_m]$. Any embedding $A \hookrightarrow \FF_q$ must send $\xi_n$ to an element of $\FF_q^*$ of order $n$ so $n$ must divide $|\FF_q^*| = q-1$. Analogously, $m$ must divide $q-1$. This means that for these finite fields we are forced to have $\xi_{nm}^{\FF_q} = nm$, $\xi_n^{\FF_q} = n$ and $\xi_m^{\FF_q} = m$, 
so the number of points is $|\Rep{m,n}(\mathrm{AGL}_{1}(\FF_{q}))| = ((n-1)(m-1)+1)(q^2-q)$. This shows that $\Rep{m,n}(\AGLC{1})$ is strongly polynomial count with counting polynomial $P(t) = ((n-1)(m-1)+1)(t^2-t)$.
\end{proof}

\begin{rmk}
The previous counting polynomial agrees with the $E$-polynomial of the representation variety in the case $\bk = \CC$, as claimed in Theorem \ref{thm:hausel-counting}. If $\bk = \CC$ (indeed if $\bk$ is algebraically closed) then $\xi_l^{\CC} = l$ so 
\begin{align*}
	e\left(\Rep{m,n}(\AGLC{1}) \right) &= ( \xi_{nm}^{\CC}-\xi_{n}^{\CC}-\xi_{m}^{\CC} + 2)(q^2-q) = ((n-1)(m-1)+1)(q^2-q) = P(q).
\end{align*}
\end{rmk}

\begin{rmk}
To show the erratic behavior of the number of points of $\Rep{m,n}(\AGL{1})$ over finite fields, in Figure \ref{fig:counting} we depict the number of points of $\Rep{m,n}(\mathrm{AGL}_{1}(\FF_q))$ for the prime powers $2 \leq q \leq 3000$. As we can observe, there appear several trends of points, which represent several `counting polynomials' for this variety. The number of trends seems to increase when $n$ and $m$ are composed numbers. 

	\begin{figure}[h!]
	\begin{center}
	\begin{subfigure}{.45\textwidth}
	\includegraphics[scale=0.47]{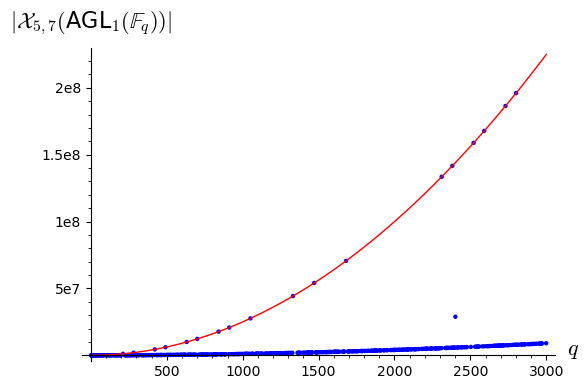}
	\subcaption[]{$(m,n) = (4,5)$}
	\end{subfigure}
	\begin{subfigure}{.45\textwidth}
	\includegraphics[scale=0.47]{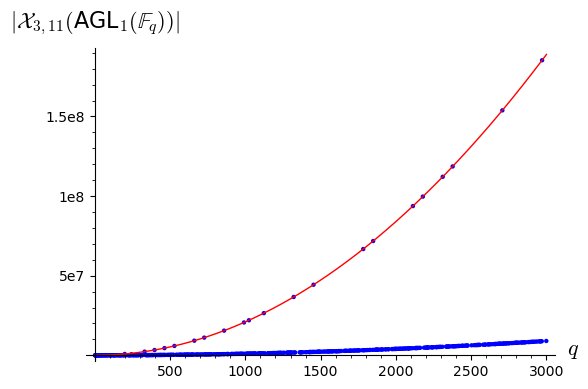}
	\subcaption[]{$(m,n) = (3,11)$}
	\end{subfigure}
	
	\vspace{0.3cm}
	
	\begin{subfigure}{.45\textwidth}
	\includegraphics[scale=0.47]{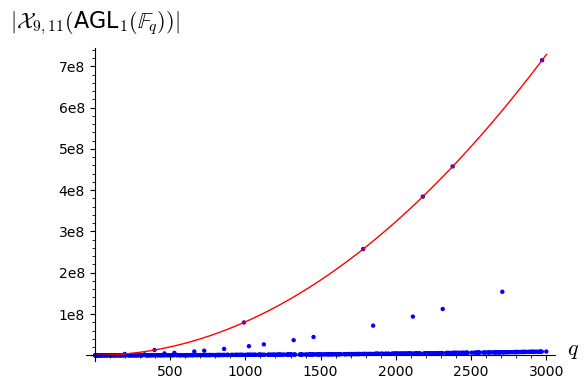}
	\subcaption[]{$(m,n) = (9,11)$}
	\end{subfigure}
	\begin{subfigure}{.45\textwidth}
	\includegraphics[scale=0.47]{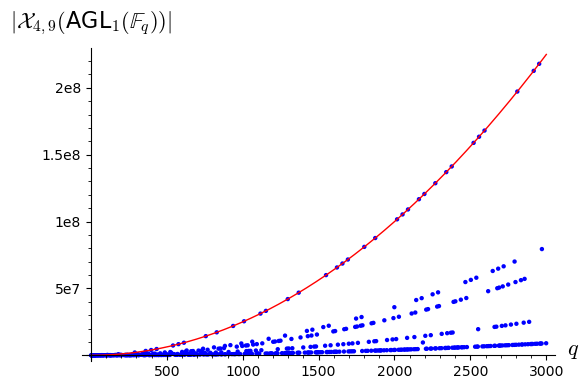}
	\subcaption[]{$(m,n) = (4,9)$}
	\end{subfigure}
	\caption{Number of points of the representation variety $\Rep{m,n}(\mathrm{AGL}_{1}(\FF_q))$ for several values of $(m,n)$ and $q$ the prime powers $q \leq 3000$. In red, the true counting polynomial that agrees with the $E$-polynomial.}
	\label{fig:counting}
	\end{center}
		\end{figure}
		
Notice that this erratic behavior poses a philosophical question. Suppose that, for some reason, we know that $\Rep{m,n}(\AGLC{1})$ 
is polynomial counting. In that case, the degree of the $E$-polynomial can be estimated from the dimension of $\Rep{m,n}(\AGLC{1})$. 
Therefore, an easy way of computing the $E$-polynomial is just to start counting points in small finite fields until we get 
enough interpolating points, which should determine the $E$-polynomial. This approach has been used for instance in \cite{GPLM}. However, the previous plot shows that this is not a good idea: most of the prime finite fields give the wrong answer, and only a few (in terms of density) prime powers can be used to interpolate the $E$-polynomial.

Moreover, even though we were able to compute the number of points of $\Rep{m,n}(\mathrm{AGL}_{1}(\FF_q))$ for any prime power $q$, say through arithmetic arguments, it is not clear how to extract the $E$-polynomial from these computations. The computations will only show that the number of points of the variety follows several trends, but the right trend is unknown beforehand (and it is not highlighted by density arguments). The selection rule for the right trend is much deeper and hard to discover: there exists a finitely generated subring of $\CC$ that is contained exactly in the finite fields of a single trend. Only one trend can satisfy this property, and that one is the correct trend.
\end{rmk}

In this spirit, we cannot expect asymptotic polynomial count to be a property that can be read from the motive. However, there is still a hope with strongly polynomial count. This poses the following conjecture which, to our knowledge, is still open. 

\begin{conjecture}\label{conj:motivic-strong}
Let $\cX$ be a reduced $\ZZ$-scheme and let $X = \cX(\CC)$ be the associated complex algebraic variety. Then $X$ is strongly polynomial count if and only if it is polynomial motivic, and the counting and motivic polynomials agree.
\end{conjecture}

\section{$\AGL{2}$-representation varieties of torus knots}\label{sec:AGL2}

In this section, we compute the motive of the $\AGL{2}$-representation variety of the $(m,n)$-torus knot. Suppose that we have an element $(A, B) \in \Rep{m,n}(\AGL{2})$ with matrices of the form
$$
	A = \begin{pmatrix}
	1 & 0 \\
	\alpha & A_0
	\end{pmatrix}, \qquad B = \begin{pmatrix}
	1 & 0 \\
	\beta & B_0
	\end{pmatrix},
$$
where $A_0, B_0 \in \GL{2}$ while $\alpha, \beta \in \bk^2$. Computing the powers we obtain
$$
	A^n = \begin{pmatrix}
	1 & 0 \\
	\Phi_n(A_0)\alpha & A^n_0
	\end{pmatrix}, \qquad B^m = \begin{pmatrix}
	1 & 0 \\
	\Phi_m(B_0)\beta & B^m_0
	\end{pmatrix}.
$$
Therefore, the $\AGL{2}$-representation variety is explicitly given by
\begin{equation}\label{eq:AGL2-rep}
\Rep{m,n}(\AGL{2}) = \left\{(A_0, B_0, \alpha, \beta) \in \GL{2}^2 \times \bk^2 \,\left|\, \begin{matrix} A_0^n = B_0^m \\ \Phi_n(A_0)\alpha = \Phi_m(B_0)\beta \end{matrix}\right.\right\}.
\end{equation}

In particular, these conditions imply that $(A_0, B_0) \in \Rep{m,n}(\GL{2})$. Let us decompose the variety as
$$
	\Rep{m,n}(\AGL{2}) = \Rep{m,n}^{\irr}(\AGL{2}) \sqcup \Rep{m,n}^{\red}(\AGL{2}).
$$
Here, $\Rep{m,n}^{\irr}(\AGL{2})$ (resp.\ $\Rep{m,n}^{\red}(\AGL{2})$) are the representations $(A, B)$ with $(A_0, B_0)$ an irreducible (resp.\ reducible) representation in the algebraic closure $\overline{\bk}$ of $\bk$.
Note that  the superscripts refer to the reducibility/irreducibility of the vectorial part of the representation, not to the representation itself. 

%
%
%
%

First of all, let us analyze the case where $(A_0, B_0)$ is an irreducible representation. In that case, the eigenvalues are restricted as the following result shows.

\begin{lem}
Let $\rho = (A_0, B_0) \in \Rep{m,n}^{\irr}(\GL{r})$ be an irreducible representation over $\overline{\bk}$. 
Then $A_0^n = B_0^m = \omega\, \mathrm{Id}$,  for some $\omega \in \bk^*$.
\end{lem}

\begin{proof}
Notice that $A^n_0$ is a linear map that is equivariant with respect to the representation $\rho$. By Schur 
lemma applied to the algebraic closure $\overline{\bk}$, this implies that
$A_0^n$ must be a multiple of the identity, say $A_0^n = \omega \, \mathrm{Id}$ for some $\omega \in \overline{\bk}$ and, since $B_0^m = A_0^n$, also $B_0^m=
\omega\, \mathrm{Id}$. However, $A_0^n$ has coefficients in $\bk$ so $\omega\in \bk^*$.
\end{proof}

\begin{cor}\label{cor;1}
Let $\rho = (A_0, B_0) \in \Rep{m,n}^{\irr}(\GL{r})$ be an irreducible representation over $\overline{\bk}$
and let $\lambda_1, \ldots, \lambda_r$ and
$\eta_1, \ldots, \eta_r$ be the eigenvalues of $A_0$ and $B_0$, respectively. If $\charF(\bk)$ does not divide $n$ and $m$, then $A_0$ and $B_0$ are diagonalizable and
$\lambda_1^n = \ldots = \lambda_r^n = \eta_1^m = \ldots = \eta_r^m$. 

Furthermore, if $r = 2$, $\charF(\bk) \neq 2$, and $m,n$ are odd, then $\lambda_i, \eta_j\in \bk^*$.
\end{cor}

\begin{proof}
The first statement is clear since the $l$-th power of non-diagonalizable matrix is non-diagonalizable provided that $\charF(\bk)$ does not divide $l$. For the second, note that $\lambda_i$ is a root of a polynomial with coefficients in $\bk$, so either $\lambda_i\in \bk^*$ or there is a quadratic extension $\bk \subset \bk(\lambda_i)$. Since $\bk$ has not characteristic $2$, this implies that $\lambda_i = a + b$ with $a \in \bk$ and $b^2 \in \bk$. Writing $n = 2r+1$ we have that
$$
	\lambda_i^n = \sum_j a^{n-i} b^{i} = \sum_{j=0}^r a^{n-2j} b^{2j} + b\left(\sum_{j=0}^r a^{n-2j-1} b^{2j}\right).
$$
Since $\lambda_i^n, a, b^{2j} \in \bk$, it follows that $b \in \bk$ and thus $\lambda_i \in \bk$.
\end{proof}

In an analogous way, we have the following result.

\begin{cor}\label{cor;2}
Suppose $m,n$ are odd coprime numbers, $\charF(\bk) \neq 2$, and $\charF(\bk)$ does not divide $n,m$. Let $\rho = (A_0, B_0) \in \Rep{m,n}^{\red}(\GL{2})$ be a reducible representation over $\overline{\bk}$. Then the eigenvalues $\lambda_i, \eta_j\in \bk^*$, and the eigenvectors are defined over 
$\bk$. In particular, $\rho$ is reducible over $\bk$.
\end{cor}

Due to the previous results, from now on we shall assume the hypotheses on $n,m$ and $\charF(\bk)$ of Corollary \ref{cor;2}.

\subsection{The irreducible stratum}\label{sec:irred-strat}

In order to analyze the conditions of (\ref{eq:AGL2-rep}), observe that $(A, B) \mapsto (A_0, B_0)$ defines a morphism
\begin{equation}\label{eq:map-fibration}
	\Rep{m,n}^{\irr}(\AGL{2}) \longrightarrow \Rep{m,n}^{\irr}(\GL{2}).
\end{equation}
The fiber of this morphism at $(A_0, B_0)$ is the kernel of the map
\begin{equation}\label{eq:map-lambda}
	\Lambda: \bk^2 \times \bk^2 \to \bk^2, \quad \Lambda(\alpha, \beta) = \Phi_n(A_0)\alpha - \Phi_m(B_0)\beta.
\end{equation}

The following appears in  \cite[Proposition 7.3]{MP}{} for $\bk=\CC$. Recall from Example \ref{ex:calculations-kvar} that $[\PGL{2}]=q^3-q$.

\begin{prop} \label{prop:GL2}
For the torus knot of type $(m,n)$, with $m,n$ both odd,
we have
$$
	[\Rep{m,n}^{\irr}(\GL{2})] = \frac14[\PGL{2}] |\Omega_{m,n}^{\bk}| (q-2) (q-1).
$$
\end{prop}

\begin{proof}
 Let $(A_0,B_0)\in \Rep{m,n}^{\irr}(\GL{2})$. Let $\lambda_1,\lambda_2$ be the eigenvalues of $A_0$ and 
$\eta_1,\eta_2$ be the eigenvalues of $B_0$. Note that $\lambda_1^n=\lambda_2^n=\eta_1^m=\eta_2^m=\omega$.
By Lemma \ref{lem:iso-cusp}, we can write $\lambda_1=t^m$, $\eta_1=t^n$, where $\omega=t^{nm}$, $t\in \bk$. Also write $\lambda_2=u^m$, $\eta_2=u^n$, with $\omega=u^{nm}$. Therefore,
$u=\varepsilon t$, with $\varepsilon\in \Omega_{m,n}^{\bk}=\mu_{nm}^{\bk}-(\mu_n^{\bk} \cup \mu_m^{\bk})$.
This follows since $\lambda_1\neq \lambda_2$ and $\eta_1\neq \eta_2$, hence $\varepsilon^n\neq 1$, $\varepsilon^m \neq 1$.

Now the ordering of $\lambda_1,\lambda_2$ and the ordering of $\eta_1,\eta_2$ produces an action of $\ZZ_2\x \ZZ_2$
on $\Omega^\bk_{n,m}$. This action is necessarily 
free. Therefore the quotient space is given by $\frac14 |\Omega_{m,n}^{\bk}|$ elements. For each of them, we have
to select points in the space (\ref{eqn:P1-}) below, hence accounting for $q-2$ and thus the result follows.
\end{proof}

To understand the kernel of (\ref{eq:map-lambda}), we use the following lemma.

\begin{lem} \label{lem:lem}
Let $A$ be a diagonalizable matrix and let $P(x) \in \CC[x]$ be a polynomial. Then the dimension of the kernel of the matrix $P(A)$
is the number of eigenvalues of $A$ that are roots of $P(x)$.
\end{lem}

\begin{proof}
Write $A = QDQ^{-1}$ with $D = \textrm{diag}(\lambda_1, \ldots, \lambda_r)$ a diagonal matrix. Then $P(A) = QP(D)Q^{-1}$ and,
since $P(D) = \textrm{diag}(P(\lambda_1), \ldots, P(\lambda_r))$, 
the dimension of its kernel is the number of eigenvalues that are also roots of $P$.
\end{proof}

Recall that we are assuming that the characteristic of $\bk$ does not divide $n,m$.
Using Lemma \ref{lem:lem}, we get that the dimension of the kernel of $\Phi_n(A_0)$ is the number of eigenvalues of $A_0$
that belong to $\hat{\mu}_n^\bk= \mu_n^\bk - \{1\}$, 
and analogously for $\Phi_m(B_0)$. Let $\lambda_1, \lambda_2$ be the eigenvalues of $A_0$ and $\eta_1, \eta_2$
the eigenvalues of $B_0$. Recall that $\lambda_1 \neq \lambda_2$ and $\eta_1 \neq \eta_2$ since otherwise $(A_0, B_0)$ is not irreducible.
Then, we have the following options:

\begin{enumerate}

\item Case $\lambda_1, \lambda_2 \in \hat{\mu}_n^\bk$ and $\eta_1, \eta_2 \in \hat{\mu}_m^\bk$. In this situation, $\Lambda \equiv 0$ so
$\Ker{\Lambda} = \bk^4$. Hence, if we denote by $\Rep{m,n}^{\irr,(1)}(\AGL{2})$ and $\Rep{m,n}^{\irr,(1)}(\GL{2})$ 
the corresponding strata in (\ref{eq:map-fibration}) of the total and base space, respectively, we have that
	$$
 \left[\Rep{m,n}^{\irr,(1)}(\AGL{2})\right] = \left[\Rep{m,n}^{\irr,(1)}(\GL{2})\right][\bk^4].
	$$	
To get the motive of $\Rep{m,n}^{\irr,(1)}(\GL{2})$, the eigenvalues define a fibration
\begin{equation}\label{eq:eigenvalues}
	\Rep{m,n}^{\irr,(1)}(\GL{2}) \longrightarrow ((\hat{\mu}_n^\bk)^2 - \Delta)/\ZZ_2 \times ((\hat{\mu}_m^\bk)^2 - \Delta)/\ZZ_2,
\end{equation}
where $\Delta$ is the diagonal and $\ZZ_2$ acts by permutation of the entries. The fiber of this map is the collection of
representations $(A_0, B_0) \in \Rep{m,n}^{\irr}(\GL{2})$ with fixed eigenvalues, denoted by $\Rep{m,n}^{\irr}(\GL{2})_{0}$. An element of $\Rep{m,n}^{\irr}(\GL{2})_{0}$ 
is completely determined by the two pairs of eigenspaces of $(A_0, B_0)$ up to conjugation. Since the representation $(A_0, B_0)$
must be irreducible, these eigenspaces must be pairwise distinct. Hence, this variety is
$\Rep{m,n}^{\irr}(\GL{2})_{0} = (\PP_\bk^1)^4 - \Delta_c$,
 where $\Delta_c \subset (\PP_\bk^1)^4$ is the `coarse diagonal' of tuples with two repeated entries. There is a free and closed action of $\PGL{2}$ on $(\PP_\bk^1)^4$ with quotient 
	\begin{equation}\label{eqn:P1-}
 \frac{(\PP_\bk^1)^4 - \Delta_c}{\PGL{2}} = \PP_\bk^1-\{0,1,\infty\}.
	\end{equation}
To see this, note that there is a $\PGL{2}$-equivariant map that sends the first three entries to
$0,1,\infty \in \PP_\bk^1$ respectively, so the orbit is completely determined by the image of the fourth point under this map. 
Hence, $[\Rep{m,n}^{\irr}(\GL{2})_{0}] = [\PP_\bk^1-\{0,1,\infty\}]\, [\PGL{2}]=(q-2)(q^3-q)$. 
	
Coming back to the fibration (\ref{eq:eigenvalues}), we have that the base space is a set of 
$\binom{\xi_n^{\bk}-1}{2}\binom{\xi_m^{\bk}-1}{2}$ points, so
	$$
 \left[\Rep{m,n}^{\irr,(1)}(\GL{2})\right] = \frac{(\xi^\bk_n-1)(\xi^\bk_n-2)(\xi^\bk_m-1)(\xi^\bk_m-2)}{4}(q-2)(q^3-q),
	$$
and thus,
	$$
 \left[\Rep{m,n}^{\irr,(1)}(\AGL{2})\right] = \frac{(\xi^\bk_n-1)(\xi^\bk_n-2)(\xi^\bk_m-1)(\xi^\bk_m-2)}{4}(q^5-2q^4)(q^3-q).
	$$
	
\item Case $\lambda_1, \lambda_2 \in \hat{\mu}_n^\bk$, $\eta_1 \in \hat{\mu}_m^\bk$ and $\eta_2 = 1$. In this case, $\Ker{\Lambda} = \bk^3$
and the base space is made of $\binom{\xi^\bk_n -1}{2} (m-1)$ copies of $\Rep{m,n}^{\irr}(\GL{2})_{0}$. 
Hence, this stratum contributes
 \begin{align*}
 \left[\Rep{m,n}^{\irr,(2)}(\AGL{2})\right] &= \frac{(\xi^\bk_n-1)(\xi^\bk_n-2)(\xi^\bk_m-1)}{2}\left[\PP^1-\left\{0,1,\infty\right\}\right]\, [\PGL{2}]\,[\bk^3] \\
 &= \frac{(\xi^\bk_n-1)(\xi^\bk_n-2)(\xi^\bk_m-1)}{2}(q^4-2q^3)(q^3-q).
	\end{align*}

\item Case $\lambda_1 \in \hat{\mu}_n^\bk$, $\lambda_2 = 1$ and $\eta_1,\eta_2 \in \hat{\mu}_m^\bk$. This is analogous to the previous stratum and contributes
 \begin{align*}
 \left[\Rep{m,n}^{\irr,(3)}(\AGL{2})\right] &= \frac{(\xi^\bk_m-1)(\xi^\bk_n-1)(\xi^\bk_m-2)}{2}\left[\PP^1-\left\{0,1,\infty\right\}\right]\, [\PGL{2}]\, [\bk^3] \\
 &= \frac{(\xi^\bk_m-1)(\xi^\bk_n-1)(\xi^\bk_m-2)}{2}(q^4-2q^3)(q^3-q).
 \end{align*}

\item Case $\lambda_1 \in \hat{\mu}_n^\bk$, $\lambda_2 = 1$ and $\eta_1 \in \hat{\mu}_m^\bk$, $\eta_2 = 1$. Now, $\Ker{\Lambda} = \bk^2$ and this stratum contributes 
 \begin{align*}
  \left[\Rep{m,n}^{\irr,(4)}(\AGL{2})\right] &= (\xi^\bk_m-1)(\xi^\bk_n-1)\left[\PP^1-\left\{0,1,\infty\right\}\right]\, [\PGL{2}]\,[\bk^2] \\
  &= (\xi^\bk_m-1)(\xi^\bk_n-1)(q^3-2q^2)(q^3-q).
 \end{align*}

\item Case $\lambda_1 \not\in \hat{\mu}_n^\bk, \lambda_2 \not\in \hat{\mu}_n^\bk, \eta_1 \not\in \hat{\mu}_m^\bk$ and $\eta_2 \not\in \hat{\mu}_m^\bk$.
Recall that by Corollary \ref{cor;1}, these conditions are all equivalent.
In this situation, $\Lambda$ is surjective so $\Ker{\Lambda} = \bk^2$. The motive 
$\left[\Rep{m,n}^{\irr}(\GL{2})\right]$ is given in Proposition \ref{prop:GL2}.
To this space, we have to remove the orbits corresponding to the forbidden eigenvalues, which are 
  \begin{align*}
 \ell_{m,n} =& \, \frac{(\xi^\bk_n-1)(\xi^\bk_n-2)(\xi^\bk_m-1)(\xi^\bk_m-2)}{4} + \frac{(\xi^\bk_n-1)(\xi^\bk_n-2)(\xi^\bk_m-1)}{2} \\ &+ \frac{(\xi^\bk_m-1)(\xi^\bk_n-1)(\xi^\bk_m-2)}{2} + (\xi^\bk_m-1)(\xi^\bk_n-1) =
\frac14 \xi^\bk_m\xi^\bk_n(\xi^\bk_m-1)(\xi^\bk_n-1) 
 \end{align*} 
 copies of $[\Rep{m,n}^{\irr}(\GL{2})_{0}] = [\PP^1-\{0,1,\infty\}]\, [\PGL{2}]$. Hence this stratum contributes
  \begin{align*}
 \left[\Rep{m,n}^{\irr,(5)}(\AGL{2})\right]  =& \left(\left[\Rep{m,n}^{\irr}(\GL{2})\right] - \ell_{m,n}(q-2)(q^3-q)\right)\left[\bk^2\right]
  \\
  =& \left[\Rep{m,n}^{\irr}(\GL{2})\right] q^2-  \frac14 \xi^\bk_m\xi^\bk_n(\xi^\bk_m-1)(\xi^\bk_n-1) (q^3-2q^2)(q^3-q) \\
=&\,  \frac14 (q^3-q) (q^3-2q^2) (\xi^\bk_m-1)(\xi^\bk_n-1) \big(  q-1 -  \xi^\bk_m\xi^\bk_n\big),
 \end{align*} 
using Proposition \ref{prop:GL2} which says that
$[\Rep{m,n}^{\irr}(\GL{2})] = \frac14(q^3-q)(\xi^\bk_m-1)(\xi^\bk_n-1) (q-2) (q-1)$.

\end{enumerate}

Adding up all the contributions, we get 
\begin{align*}
	\left[\Rep{m,n}^{\irr}(\AGL{2})\right] =& \left[\Rep{m,n}^{\irr,(1)}(\AGL{2})\right] + \left[\Rep{m,n}^{\irr,(2)}(\AGL{2})\right]
	+ \left[\Rep{m,n}^{\irr,(3)}(\AGL{2})\right] \\
	& + \left[\Rep{m,n}^{\irr,(4)}(\AGL{2})\right] + \left[\Rep{m,n}^{\irr,(5)}(\AGL{2})\right]  \\
	= &\, \frac{(\xi^\bk_m-1)(\xi^\bk_n-1)(q^4-3q^3+2q^2)(q^3-q)}{4}  \left( (\xi^\bk_m-2)(\xi^\bk_n-2)q+\xi^\bk_m\xi^\bk_n-3 \right).
\end{align*}

\subsection{The reducible stratum}

Recall our assumption that $\charF(\bk)$ does not divide $n$ and $m$.
In this section, we consider the case in which $(A_0, B_0) \in \Rep{m,n}^{\red}(\GL{2})$ 
is a reducible representation. 
After a change of basis, since $A_0^n = B_0^m$, we can suppose that $(A_0, B_0)$ has exactly one of the following three forms:
 $$
 \textrm{(A)} \left(\begin{pmatrix}t_1^m & 0 \\ 0 & t_2^m\end{pmatrix}, \begin{pmatrix}t_1^n & 0 \\ 0 & t_2^n\end{pmatrix}\right), 
   \textrm{(B)} \left(\begin{pmatrix}t^m & 0 \\ 0 & t^m\end{pmatrix}, \begin{pmatrix}t^n & 0 \\ 0 & t^n\end{pmatrix}\right),
 \textrm{(C)} \left(\begin{pmatrix}t^m & 0 \\ x & t^m\end{pmatrix}, \begin{pmatrix}t^n & 0 \\ y & t^n\end{pmatrix}\right),
 $$
with $t_1, t_2, t \in \bk^*$, $x, y \in \bk$ and satisfying $t_1 \neq t_2$ and $(x,y) \neq (0,0)$.

Restricting to the representations of each stratum $S = (\textrm{A}), (\textrm{B}), (\textrm{C})$, 
we have a morphism
 \begin{equation}\label{eq:map-fibration-A}
 \Rep{m,n}^{S}(\AGL{2}) \longrightarrow \Rep{m,n}^{S}(\GL{2}),
 \end{equation}
whose fiber is the kernel of the linear map (\ref{eq:map-lambda}).

\subsubsection{Case \textrm{(A)}}\label{sec:red-A}
In this case, as for the irreducible part of Section \ref{sec:irred-strat}, the kernel of $\Lambda$ depends on whether $t_1, t_2$ are
roots of the polynomial $\Phi_l$. In this case the base space is
 $$
 \Rep{m,n}^{ \textrm{(A)}}
 (\GL{2}) = \left(\left((\bk^*)^2 - \Delta\right) \times \frac{\GL{2}}{\GL{1} \times \GL{1}}\right)/\ZZ_2,
 $$
with the action of $\ZZ_2$ given by exchange of eigenvalues and eigenvectors. Using Lemma \ref{lem;ac} and (\ref{eqn:2bis}), 
we have
 \begin{align*}
 \left[\Rep{m,n}^{\mathrm{(A)}}(\GL{2})\right] &=
 [(\bk^*)^2 - \Delta]^+ 
  \left[\frac{\GL{2}}{\GL{1} \times \GL{1}}\right]^+ 
  + [(\bk^*)^2 - \Delta]^-
   \left[\frac{\GL{2}}{\GL{1} \times \GL{1}}\right]^- \\
 &= q^2(q-1)^2 - q(q-1).
\end{align*}

On the other hand, if we fix the eigenvalues of $(A_0, B_0)$ as in Section \ref{sec:irred-strat}, 
the corresponding fiber $\Rep{m,n}^{\mathrm{(A)}}(\GL{2})_0$ is
$$
 \left[\Rep{m,n}^{\mathrm{(A)}}(\GL{2})_0\right] = \left[\frac{\GL{2}}{\GL{1} \times \GL{1}}\right] = q^2 + q.
$$

As in Section \ref{sec:AGL1}, set $\Omega^{\bk}_{m,n} = \mu_{mn}^\bk - (\mu_{m}^\bk \cup \mu_{n}^\bk)$ for those 
$t \in \bk^*$ such that $\Phi_n(t^m) = 0$ and $\Phi_m(t^n) = 0$. With this information at hand, we compute for each stratum:

\begin{enumerate}
 \item Case $t_1, t_2 \in \Omega^{\bk}_{m,n}$. In this situation, $\Lambda \equiv 0$ so $\Ker{\Lambda} = \bk^4$. 
  The eigenvalues yield a fibration
\begin{equation*}
	\Rep{m,n}^{\mathrm{(A)},(1)}(\GL{2}) \longrightarrow \left((\Omega_{m,n}^{\bk})^2 - \Delta\right)/\ZZ_2
\end{equation*}
whose fiber is $\Rep{m,n}^{\mathrm{(A)}}(\GL{2})_0$. Observe that $\left((\Omega_{m,n}^{\bk})^2 - \Delta\right)/\ZZ_2$ 
is a finite set of $(\xi^\bk_m-1)(\xi^\bk_n-1)((\xi^\bk_m-1)(\xi^\bk_n-1)-1)/2$ points, so we have
\begin{align*}
	\left[\Rep{m,n}^{\mathrm{(A)},(1)}(\AGL{2})\right] &= \left[\Rep{m,n}^{\mathrm{(A)},(1)}(\GL{2})\right][\bk^4] \\
	&= \left[\Rep{m,n}^{\mathrm{(A)}}(\GL{2})_0\right][\bk^4]\left[\left((\Omega_{m,n}^{\bk})^2 - \Delta\right)/\ZZ_2\right] \\
	&= \frac{(\xi^\bk_m-1)(\xi^\bk_n-1)(\xi^\bk_m\xi^\bk_n-\xi^\bk_m-\xi^\bk_n)}{2}q^4(q^2+q).
\end{align*}

\item Case $t_1 \in \Omega^{\bk}_{m,n}$ but $t_2 \not\in \Omega^{\bk}_{m,n}$ (or vice-versa, the order is not important here).
Now, we have a locally trivial fibration
$$
 \Rep{m,n}^{\mathrm{(A)},(2)}(\GL{2}) \longrightarrow \Omega^{\bk}_{m,n} \times \left(\bk^*-\Omega^{\bk}_{m,n}\right),
$$
with fiber $\Rep{m,n}^{\mathrm{(A)}}(\GL{2})_0$. The kernel of $\Lambda$ is $\bk^3$, so this stratum contributes
\begin{align*}
 \left[\Rep{m,n}^{\mathrm{(A)},(2)}(\AGL{2})\right] = (\xi^\bk_m-1)(\xi^\bk_n-1)(q -\xi^\bk_m\xi^\bk_n + \xi^\bk_m + \xi^\bk_n -2) q^3(q^2+q).
\end{align*}

\item Case $t_1, t_2 \not\in\Omega^{\bk}_{m,n}$. The kernel is now $\bk^2$ and we have a fibration
$$
 \Rep{m,n}^{\mathrm{(A)},(3)}(\GL{2}) \longrightarrow B,
$$
where the motive of the base space $B$ is 
\begin{align*}
 [B] &=[(\bk^*)^2-\Delta]^+  - \left[(\Omega_{m,n}^{\bk})^2 - \Delta\right]^+ 
- \left[\Omega^{\bk}_{m,n}\right] \left(q-1-[\Omega^{\bk}_{m,n}]\right) =\\
 &=  (q-1)^2 -\frac{(\xi^\bk_m-1)(\xi^\bk_n-1)(\xi^\bk_m\xi^\bk_n - \xi^\bk_m - \xi^\bk_n )}{2}\\
 & \;\;\;\;- (\xi^\bk_m-1)(\xi^\bk_n-1)(q-\xi^\bk_m\xi^\bk_n + \xi^\bk_m + \xi^\bk_n -2) \\
 &= q^2 - (\xi^\bk_m\xi^\bk_n - \xi^\bk_m - \xi^\bk_n +3)q -\frac14 (\xi^\bk_m-1)(\xi^\bk_n-1)(\xi^\bk_m\xi^\bk_n-8).
\end{align*}
Therefore, this space contributes
\begin{align*}
 \big[\Rep{m,n}^{\mathrm{(A)},(3)}&(\AGL{2})\big] = 
 \left[\Rep{m,n}^{\mathrm{(A)},(3)}(\GL{2})\right][\bk^2]  \\
 &= q^2(q^2+q)\left(q^2 - (\xi^\bk_m\xi^\bk_n - \xi^\bk_m - \xi^\bk_n +3)q+ \frac14 (\xi^\bk_m-1)(\xi^\bk_n-1)(\xi^\bk_m\xi^\bk_n-8)\right).
\end{align*}
\end{enumerate} 

Adding up all the contributions, we get that

\begin{align*}
 \left[\Rep{m,n}^{\mathrm{(A)}}(\AGL{2})\right] =\; &
 (q^2+q)q^2\bigg(\frac12 (\xi^\bk_m-1)(\xi^\bk_n-1)(\xi^\bk_m\xi^\bk_n - \xi^\bk_m - \xi^\bk_n )(q^2-1) \\
 & + (\xi^\bk_m-1)(\xi^\bk_n-1)(q - \xi^\bk_m\xi^\bk_n + \xi^\bk_m + \xi^\bk_n -2) (q-1) + (q-1)^2\bigg).
 \end{align*}

\subsubsection{Case \textrm{(B)}}\label{sec:red-B} In this setting, this situation is simpler. Observe that the adjoint
action of $\GL{2}$ on the vectorial part is trivial, so the corresponding $\GL{2}$-representation variety is just
$$
	\Rep{m,n}^{\mathrm{(B)}}(\GL{2}) = \bk^*.
$$
Analogously, the variety with fixed eigenvalues, $\Rep{m,n}^{\mathrm{(B)}}(\GL{2})_0$, is just a point.
With these observations, we obtain that:
\begin{enumerate}
\item If $t \in \Omega^{\bk}_{m,n}$, then $\Ker{\Lambda} = \bk^4$. We have a fibration
\begin{equation*}
 \Rep{m,n}^{\mathrm{(B)},(1)}(\GL{2}) \longrightarrow \Omega^{\bk}_{m,n}
\end{equation*}
whose fiber is $\Rep{m,n}^{\mathrm{(B)}}(\GL{2})_0$. Hence, this stratum contributes
$$
 \left[\Rep{m,n}^{\mathrm{(B)},(1)}(\AGL{2})\right] =
  \left[\Rep{m,n}^{\mathrm{(B)},(1)}(\GL{2})\right][\bk^4] = (\xi^\bk_m-1)(\xi^\bk_n-1)q^4\, .
$$

\item If $t \not\in \Omega^{\bk}_{m,n}$, then $\Ker{\Lambda} = \bk^2$. We have a fibration
\begin{equation*}
 \Rep{m,n}^{\mathrm{(B)},(2)}(\GL{2}) \longrightarrow \bk^*-\Omega^{\bk}_{m,n}.
\end{equation*}
Thus, the contribution of this stratum is
$$
 \left[\Rep{m,n}^{\mathrm{(B)},(2)}(\AGL{2})\right] = (q-1-(\xi^\bk_m-1)(\xi^\bk_n-1))q^2.
$$
\end{enumerate}
The total contribution is thus
 $$
  \left[\Rep{m,n}^{\mathrm{(B)}}(\AGL{2})\right] = (\xi^\bk_m-1)(\xi^\bk_n-1)(q^4-q^2)+(q-1)q^2.
$$

\subsubsection{Case (C)}\label{sec:red-C} 
In this case, an extra calculation must be done to control the off-diagonal entry. If $(A_0, B_0)$ has the form
$$ \left(\begin{pmatrix}t^m & 0 \\ x & t^m\end{pmatrix}, \begin{pmatrix}t^n & 0 \\ y & t^n\end{pmatrix}\right),
$$
then the condition $A_0^n = B_0^m$ reads as
$$
 \begin{pmatrix}t^{mn} & 0 \\ nt^{m(n-1)}x & t^{mn}\end{pmatrix} = \begin{pmatrix}t^{mn} & 0 \\ mt^{n(m-1)}y & t^{mn}\end{pmatrix}.
$$
The later conditions reduce to $nt^{m(n-1)}x = mt^{n(m-1)}y$ and, since $t \neq 0$, this means that $(x,y)$ should lie in a
line minus $(0,0)$. The stabilizer of a Jordan type matrix in $\GL{2}$ is the subgroup $U = (\bk^*)^2 \times \bk \subset \GL{2}$
of upper triangular matrices. Hence, the corresponding $\GL{2}$-representation variety is
$$
	\Rep{m,n}^{\mathrm{(C)}}(\GL{2}) = (\bk^*)^2 \times \GL{2}/U.
$$
In particular, $\left[\Rep{m,n}^{\mathrm{(C)}}(\GL{2})\right] 
= (q-1)^2 (q^4-q^3-q^2+q)/q(q-1)^2 = (q-1)^2(q+1)$. Moreover, if we fix 
the eigenvalues we get that $\left[\Rep{m,n}^{\mathrm{(C)}}(\GL{2})_0\right] = \left[\bk^* \times \GL{2}/U\right] = (q-1)(q+1)$.

To analyze the condition $\Phi_n(A_0) = \Phi_m(B_0)$, a straightforward computation reduces it to
$$
	\begin{pmatrix}\Phi_n(t^m) & 0 \\ \displaystyle{x\sum_{i=1}^{n-1} it^{m(i-1)}} & \Phi_n(t^m)\end{pmatrix} = \begin{pmatrix}\Phi_m(t^n) & 0 
	\\ \displaystyle{y\sum_{i=1}^{m-1} it^{n(i-1)}} & \Phi_m(t^n)\end{pmatrix}.
$$
The off-diagonal entries can be recognized as $x\Phi_n'(t^m)$ and $y\Phi_m'(t^n)$ respectively, where $\Phi_l'(x)$ denotes the formal
derivative of $\Phi_l(x)$. Observe that $\Phi_l(x)(x-1) = x^l-1$, whose derivative is $lx^{l-1}$, so $\Phi_l$ has no repeated roots provided that $\charF(\bk)$ does not divide $l$. Hence, with our assumptions on $\charF(\bk)$ we have that $\Phi_n(t^m), \Phi_m(t^n), \Phi_n'(t^m)$
and $\Phi_m'(t^n)$ cannot vanish simultaneously. Therefore, stratifying according to the kernel of $\Lambda$ we get the following two possibilities:

\begin{enumerate}
\item If $t \in \Omega^{\bk}_{m,n}$, then $\Ker{\Lambda} = \bk^3$. We have a fibration
\begin{equation*}
 \Rep{m,n}^{\mathrm{(C)},(1)}(\GL{2}) \longrightarrow \Omega^{\bk}_{m,n}
\end{equation*}
whose fiber is $\Rep{m,n}^{\mathrm{(C)}}(\GL{2})_0$. Hence, this stratum contributes
$$
 \left[\Rep{m,n}^{\mathrm{(C)},(1)}(\AGL{2})\right] = \left[\Rep{m,n}^{\mathrm{(C)},(1)}(\GL{2})\right][\bk^3] = (\xi^\bk_m-1)(\xi^\bk_n-1)q^3(q-1)(q+1).
$$

\item If $t \in \bk^*-\Omega^{\bk}_{m,n}$, then $\Ker{\Lambda} = \bk^2$. The fibration we get is now
\begin{equation*}
	\Rep{m,n}^{\mathrm{(C)},(2)}(\GL{2}) \longrightarrow \bk^*-\Omega^{\bk}_{m,n}\, .
\end{equation*}
Therefore, this stratum contributes
\begin{align*}
 \left[\Rep{m,n}^{\mathrm{(C)},(2)}(\AGL{2})\right] &= \left[\Rep{m,n}^{\mathrm{(C)},(2)}(\GL{2})\right][\bk^2] \\ 
 &= \left[\Rep{m,n}^{\mathrm{(C)}}(\GL{2}) - (\xi^\bk_m-1)(\xi^\bk_n-1)\Rep{m,n}^{\mathrm{(C)},(2)}(\GL{2})_0\right][\bk^2] \\
 &= \left((q-1)^2(q+1) - (\xi^\bk_m-1)(\xi^\bk_n-1)(q-1)(q+1)\right)q^2.
\end{align*}
\end{enumerate}

Adding up all the contributions, we get that
$$
 \left[\Rep{m,n}^{\mathrm{(C)}}(\AGL{2})\right] = (q-1)^2(q+1)q^2 + (\xi^\bk_m-1)(\xi^\bk_n-1)(q-1)(q+1)(q^3-q^2).
$$

Putting the results of Sections \ref{sec:irred-strat}, \ref{sec:red-A}, \ref{sec:red-B} and \ref{sec:red-C} together, we prove the
second formula in Theorem \ref{thm:main}.

\end{document}